\documentclass[10pt]{amsart}

\usepackage{hyperref}
\usepackage{enumerate}
\usepackage{comment}

\makeatletter

\@namedef{subjclassname@2010}{

  \textup{2010} Mathematics Subject Classification}

\usepackage{amssymb, amsmath,amsthm}
\usepackage{mathrsfs}
\newtheorem{thm}{Theorem}[section]

\newtheorem{prop}[thm]{Proposition}
\newtheorem{conj}[thm]{Conjecture}

\newtheorem{cor}[thm]{Corollary}
\newtheorem{lem}[thm]{Lemma}
\theoremstyle{definition}
\newtheorem{rem}[thm]{Remark}

\newcommand{\ra}{\rightarrow}
\newcommand{\bk}{\backslash}
\newcommand{\mc}{\mathcal}

\newcommand{\mb}{\mathbb}
\newcommand{\sg}{\sigma}

\renewcommand{\ss}{\substack}

\renewcommand{\bar}{\overline}

\begin{document}
\title{On a conjecture of Goldmakher}
\author{Alexander P. Mangerel}
\address{Department of Mathematical Sciences, Durham University, Stockton Road, Durham, DH1 3LE, UK}
\email{smangerel@gmail.com}
\begin{abstract}
We construct a $1$-bounded completely multiplicative function $f$ whose logarithmically-averaged partial sums satisfy
$$
\limsup_{x \ra \infty} \frac{\left|\sum_{n \leq x} \frac{f(n)}{n}\right|}{1+\exp\left(\sum_{p \leq x} \frac{\text{Re}(f(p))}{p}\right)}
= \infty.
$$
This disproves a conjecture of Goldmakher from 2009. 
\end{abstract}
\maketitle
\section{Main Results}
Let $\mc{F}$ denote the collection of all $1$-bounded, completely multiplicative functions. For $f \in \mc{F}$ we define
$$
L_f(x) := \sum_{n \leq x} \frac{f(n)}{n}, \quad x \geq 1.
$$
Given $f,g \in \mc{F}$ and $z \geq 1$ we denote by
$$
\mb{D}(f,g;z) := \left(\sum_{p \leq z} \frac{1-\text{Re}(f(p)\bar{g(p)})}{p}\right)^{1/2},
$$
the \emph{pretentious distance} between $f$ and $g$ up to $z$. \\
In \cite{Gold}, Goldmakher conjectured the following bound, relating the size of the logarithmic partial sums of $f$ to its pretentious distance from the constant function $1$.
\begin{conj}[\cite{Gold}, Conj. 2.6] \label{conj:Gold}
Let $f \in \mc{F}$. Then for any $1 \leq y \leq x$ we have
$$
\sum_{\ss{n \leq x \\ p|n \Rightarrow p \leq y}} \frac{f(n)}{n} \ll 1 + (\log y)e^{-\mb{D}(f,1;y)^2}.
$$
\end{conj}
Morally, Goldmakher's conjecture asserts that the only way for $L_f(x)$ to be ``large'' (or even \emph{unbounded} with $x$) is if $f$ ``pretends'' to be the constant function $1$, i.e., $\mb{D}(f,1;x)^2$ does not grow too fast with $x$. This is in contrast to the situation for ordinary (unweighted) partial sums 
$$
\tilde{M}_f(x) := \frac{1}{x}\sum_{n \leq x} f(n),
$$
for which it is known that $\tilde{M}_f(x) \gg 1$ whenever $f \in \mc{F}$ ``pretends'' to be $n^{it}$ for \emph{some} fixed $t \in \mb{R}$. \\
Estimates of the kind given by Conjecture \ref{conj:Gold} (especially in the case $y = x$) are sought after in the literature on multiplicative functions, as obtaining quantitative bounds on the size of $L_f(x)$ is made easier if all that is required is to know how close $f$ is to $1$ (with respect to pretentious distance), rather than to $n^{it}$ with some uniformity in a range of $t \in \mb{R}$.\\
A significant, well-known example of this arises from the work of Hall and Tenenbaum \cite{HT} and Hall \cite{Hall} (see also the earlier works of Elliott \cite{Ell} and Hal\'{a}sz \cite{Hal}). They proved several results concerning the best possible constant $\kappa = \kappa_f > 0$ that can be taken so that
$$
\left|\tilde{M}_f(x) \right| \ll e^{-\kappa \mb{D}(f,1;x)^2}.
$$
The example $f(n) = n^{it}$ with $t \neq 0$ shows that $\inf_{f \in \mc{F}} \kappa_f = 0$ (as in this case, $\mb{D}(n^{it},1;x)^2 \ra \infty$ by virtue of $\zeta(1+it) \neq 0$); however when the infimum is taken over 
$$
\mc{F}_{\Omega} := \{f \in \mc{F}: \ f(p) \in \Omega \text{ for all primes } p \},
$$
for certain convex sets $\Omega \subseteq \{z \in \mb{C} : \ |z| \leq 1\}$ containing $0$ it is possible for 
$$
\kappa_{\Omega} := \inf_{f \in \mc{F}_{\Omega}}\kappa_f > 0.
$$ 
As a well-known special case, when $\Omega = [-1,1]$, i.e., when restricting to $1$-bounded, real-valued functions, one has $\kappa_{[-1,1]} = 0.32867...$ (the value of $-\cos \phi_0$, where $\phi_0$ is the unique root of the equation $\sin \phi_0 - \phi_0 \cos \phi_0 = \pi/2$). Rather definitive results were obtained in far more generality in \cite{Hall}. In many cases, however, one can also show that $\kappa_{\Omega} < 1$.\\
 In the case of logarithmic partial sums, the situation is rather different. When $f \in \mc{F}$ is real-valued, for example, it is a simple consequence of \eqref{eq:triv} below and the non-negativity of the multiplicative function $1\ast f$ that
 $$
 L_f(x) \ll 1 + (\log x) e^{-\mb{D}(f,1;x)^2},
 $$
thus confirming Conjecture \ref{conj:Gold} for the subcollection $\mc{F}_{[-1,1]}$ (see also \cite[Cor. 2.7]{NegTrunc}, where the secondary term $1$ is refined to $(\log x)^{2/\pi - 1}$). \\
When $f$ is not real-valued the situation is less clear. Since $L_{n^{it}}(x)$ is well-approximated by $\zeta(1+1/\log x - it)$, which is rather small compared to $\log x$ when $t \neq 0$ is fixed and $x \ra \infty$, one does not expect the distance from $f$ to $n^{it}$ for $|t| \geq 1$, say, to influence the bounds on $L_f(x)$ if this is unbounded. Indeed, as a special case of \cite[Thm. 1.4]{LamMan}, Lamzouri and the author showed that for any $f \in \mc{F}$ and $2 \leq y \leq x$, 
$$
\sum_{\ss{n \leq x \\ p|n \Rightarrow p \leq y}} \frac{f(n)}{n} \ll 1 + (\log y) e^{-\min_{|t| \leq 1} \mb{D}(f,n^{it};y)^2},
$$
thus giving a weak variant of Conjecture \ref{conj:Gold}. \\
As far as bounds given only in terms of $\mb{D}(f,1;x)^2$, a first result in this direction was obtained by Granville and Soundararajan, as a key ingredient in their breakthrough work on improvements to the P\'{o}lya-Vinogradov inequality for odd order characters \cite{GSPret}. They showed that for all $f \in \mc{F}$ we have the weaker inequality
\begin{equation}\label{eq:GShalf}
L_f(x) \ll 1 + (\log x) e^{-\tfrac{1}{2}\mb{D}(f,1;x)^2}.
\end{equation}
The improvements in \cite{Gold} on the odd order character sum problem motivate the speculation of Conjecture \ref{conj:Gold} that $1/2$ can be improved to $1$ in \eqref{eq:GShalf}. \\
In \cite[Prop. 1.2]{GraMan}, Granville and the author showed that
\begin{equation}\label{eq:genUppBd}
L_f(x) 
\ll (\log\log x)\left(1+(\log x)e^{-\lambda \mb{D}(f,1;x)^2}\right),
\end{equation}
where $\lambda = 0.8221...$ (the solution to an explicit integral equation), and that moreover $\lambda$ is best possible in that for any large scale $x$ there exists a function $f = f_x \in \mc{F}$ such that
\begin{equation}\label{eq:optEx}
L_f(x) \asymp (\log x)\exp\left(-\lambda \mb{D}(f,1;x)^2\right).
\end{equation}
The latter results already suggest that the estimate in Conjecture \ref{conj:Gold} (with exponent $1$ in place of $\lambda$) is not correct. However, the estimate \eqref{eq:genUppBd} is, in many cases, weaker than Goldmakher's conjecture (for example, whenever $|L_f(x)|$ is bounded). Also, the example furnishing the optimality of $\lambda$ in \eqref{eq:optEx} satisfies $L_f(x) \ll 1$, and is therefore not a counterexample to Conjecture \ref{conj:Gold}. \\
The purpose of this note is to definitively disprove Goldmakher's conjecture.
\begin{thm}\label{thm:infScales}
There exists $f \in \mc{F}$ such that
$$
\limsup_{x \ra \infty} \frac{|L_f(x)|}{1+(\log x)e^{-\mb{D}(f,1;x)^2}}
= \infty.
$$
Thus, Conjecture \ref{conj:Gold} is false.
\end{thm}
Key to the construction of the counterexample given in Theorem \ref{thm:infScales} is the following finite scale version.
\begin{thm} \label{thm:counter}
Let $C > 0$. Then there is $x_0(C)$ such that for every $x \geq x_0(C)$ there is a function $f = f_x \in \mc{F}$ such that
$$
L_f(x) \geq C\left(1+(\log x)e^{-\mb{D}(f,1;x)^2}\right).
$$
\end{thm}
\subsection{Proof Idea}
In the sequel, for any arithmetic function $h$ and $x \geq 1$ we define
$$
M_h(x) := \sum_{n \leq x} h(n), \quad \tilde{M}_h(x) := \frac{1}{x}M_h(x).
$$
In light of the elementary estimate
\begin{equation}\label{eq:triv}
L_f(x) = \tilde{M}_{1\ast f}(x) + O(1),
\end{equation}
as well as Mertens' theorem, we may obtain the following consequence of Goldmakher's conjecture, in the case $y=x$; it is this statement that we shall disprove in the sequel.
\begin{conj}[Consequence of Conjecture \ref{conj:Gold}] \label{conj:altGold}
Let $f \in \mc{F}$ and set $g := 1\ast f$. Then for any $x \geq 1$,
\begin{equation}\label{eq:Goldforh}
\tilde{M}_{g}(x) \ll 1 + \exp\left(\sum_{p \leq x} \frac{\text{Re}(f(p))}{p}\right).
\end{equation}
\end{conj}
Here and elsewhere, for $f \in \mc{F}$ and $s \in \mb{R}$ with $\text{Re}(s) > 1$ we write
$$
L(s,f) := \sum_{n \geq 1} \frac{f(n)}{n^s}.
$$
Now, according to Perron's formula the LHS of \eqref{eq:Goldforh} depends on the sizes of the peaks of the function
$$
\frac{L(\sg + it,f) \zeta(\sg + it)}{\sg + it} x^{\sg + it}, \quad \sg = 1+1/\log x, \ t \in \mb{R}.
$$
We already expect any such peaks to occur for $|t| \leq 1$, where $|\zeta(\sg + it)|$ is largest. Since
$$
|L(\sg,f)| \asymp \exp\left(\sum_{p \leq x} \frac{Re(f(p))}{p}\right)
$$
(as in \eqref{eq:Mertcons} below), Conjecture \ref{conj:altGold} implicitly suggests that the Perron integral will be dominated by a peak at $t = 0$. 
Our counterexample emerges from investigating whether other possible peaks can occur. In particular, to prove Theorem \ref{thm:counter}, given a large scale $x$ we select a suitable $\xi = \xi(x) \in [-1,1]$ (in fact, $\xi \ra 0$ as $x \ra \infty$) and construct a function $f = f_x \in \mc{F}$ such that
\begin{equation}\label{eq:LfncompIntro}
|L(\sg + i \xi,f)\zeta(\sg + i\xi)| > (\sg-1)^{-c}|L(\sg,f)\zeta(\sg)|,
\end{equation}
for some absolute constant $c > 0$. The additional factor $(\sg-1)^{-c}$ is crucial for us in establishing $\sg + i\xi$ as a dominant peak in the Perron integral representing $\tilde{M}_{1\ast f}(x)$ (or more precisely in an $L^2$-integral related to Bessel's inequality for the Fourier transform of $\tilde{M}_{1\ast f}(e^v) e^{-(\sg-1) v}$; see Section \ref{sec:proof} below). \\
We deduce Theorem \ref{thm:infScales} by patching together in a compatible way the counterexamples produced by the proof of Theorem \ref{thm:counter}, at a suitably chosen infinite increasing sequence of scales; see Section \ref{sec:infScales} for the details.
\section{Construction of a counterexample}
Let $x$ be a large scale and set $\sg := 1+1/\log x$. As in the discussion of the previous section, our goal is to choose $\xi = \xi(x) \in [-1,1]$ and $f = f_x \in \mc{F}$ such that an estimate of the shape \eqref{eq:LfncompIntro} holds. As a first step we investigate sufficient conditions required to ensure the validity of the more modest inequality
\begin{equation}\label{eq:Lfncomp}
|L(\sg + i \xi,f)\zeta(\sg + i\xi)| > |L(\sg,f)\zeta(\sg)|.
\end{equation}
Let us assume that such a pair $(\xi,f)$ exists, in order to determine what properties it must satisfy. \\
As a standard consequence of Mertens' theorem, for any $h \in \mc{F}$ and $t \in \mb{R}$ we have
\begin{equation}\label{eq:Mertcons}
|L(\sg + it,h)| \asymp \exp\left(\sum_p \frac{\text{Re}(h(p)p^{-it})}{p^{\sg}}\right) \asymp \exp\left(\sum_{p \leq x} \frac{\text{Re}(h(p)p^{-it})}{p}\right).
\end{equation}
Applying this with $h \in \{1,f\}$ and $t \in \{0,\xi\}$, then taking logarithms of both sides, we obtain from \eqref{eq:Lfncomp} that
$$
\sum_{p \leq x} \frac{\text{Re}(f(p)(p^{-i\xi}-1))}{p} > \sum_{p \leq x} \frac{\text{Re}(1-p^{-i\xi})}{p} + O(1),
$$
or equivalently, that
$$
\sum_{p \leq x} \frac{\text{Re}((1-p^{-i\xi})(1+f(p)))}{p} + O(1) < 0.
$$
Using 
$$
\text{Re}(z\bar{w}) = \text{Re}(z)\text{Re}(w) + \text{Im}(z)\text{Im}(w)
$$ 
we may thus obtain \eqref{eq:Lfncomp} by ensuring that for \emph{most} primes $p$,
\begin{equation}\label{eq:toCheck}
\text{Im}(f(p))\sin(\xi\log p) > (1-\cos(\xi\log p))(1+\text{Re}(f(p))),
\end{equation}
and that the weaker \emph{non-strict} inequality holds for\footnote{We remark that whenever $f \in \mc{F}$ the RHS is non-negative, and so such a strict inequality is not necessarily possible for all $p$, e.g. if $\sin(\xi \log p) = 0$ for some $p$.} \emph{all} $p$. \\
Fix $\delta \in (0,1/2)$ to be small, and $\xi \in [-1,1]$. In the sequel, to simplify notation let us write\footnote{In the sequel we denote the set of primes by $\mb{P}$.}
$$
\sg_p := \sin(\xi \log p), \quad \gamma_p := \cos(\xi \log p), \quad p \in \mb{P}.
$$
Define the functions $R = R_{\xi,\delta}$ and $I = I_{\xi,\delta}$ on primes via
\begin{equation}\label{eq:fDef}
R(p) := \begin{cases} (1-\delta) \gamma_p &\text{ if } \gamma_p \geq 0, \\ (1+\delta) \gamma_p &\text{ if } 0 > \gamma_p > -\frac{1}{1+\delta}, \\ \gamma_p &\text{ otherwise}; \end{cases} \quad I(p) := \text{sign}(\sg_p)(1-R(p)^2)^{1/2}
\end{equation}
(note that $I(p)$ is well-defined since it is easily seen that $|R(p)| \leq 1$ for all $p$). Finally, define a completely multiplicative function $f = f_{\xi,\delta}$ at primes via 
$$
f(p) := R(p) + i I(p), \quad p \in \mb{P}.
$$
\begin{prop} \label{prop:consf}
With the above notation, the function $f = f_{\xi,\delta}$ satisfies the following properties: 
\begin{enumerate}
\item We have $|f(p)| = 1$ for all primes $p$, thus $f \in \mc{F}$.
\item For every $p$ we have
\begin{equation}\label{eq:goalp}
\text{Re}((1-p^{-i\xi})(1+f(p))) \leq 0,
\end{equation}
with equality if and only if
$$
\gamma_p \in \{0,1\} \cup [-1,-1/(1+\delta)].
$$
\end{enumerate}
\end{prop}
\begin{proof}
The first item is obvious, since $|f(p)|^2 = R(p)^2+ I(p)^2 = 1$ for all $p$. \\
We next prove the second item. As in \eqref{eq:toCheck}, we may rewrite \eqref{eq:goalp} as
\begin{equation}\label{eq:realpart}
I(p)\sg_p \geq (1-\gamma_p)(1+R(p)).
\end{equation}
We note that the LHS is $|\sg_p|(1-R(p)^2)^{1/2} \geq 0$, and as $|\gamma_p|,|R(p)| \leq 1$ the RHS is also patently non-negative. Therefore, it suffices to verify the equivalent inequality
\begin{equation}\label{eq:goal2}
I(p)^2 \sg_p^2 \geq (1-\gamma_p)^2(1+R(p))^2.
\end{equation}
If $\sg_p = 0$ then either $\gamma_p = 1$ or $\gamma_p = -1$. Clearly both sides are $0$ in the former case, while in the latter case the same is true because then $R(p) = \gamma_p = -1$. Thus,  \eqref{eq:goal2} must hold whenever $\sg_p = 0$. 
Let us assume henceforth that $\sg_p \neq 0$ (and thus $R(p) \neq -1$ as well). 
Now, we clearly have
$$
\frac{I(p)^2}{(1+R(p))^2} = \frac{1-R(p)^2}{(1+R(p))^2} = \frac{1-R(p)}{1+R(p)},
$$
and similarly
$$
\frac{(1-\gamma_p)^2}{\sg_p^2} = \frac{(1-\gamma_p)^2}{1-\gamma_p^2} = \frac{1-\gamma_p}{1+\gamma_p}.
$$
Therefore, \eqref{eq:goal2} is equivalent to
$$
\frac{1-R(p)}{1+R(p)} \geq \frac{1-\gamma_p}{1+\gamma_p}.
$$
or indeed, equivalently 
$$
R(p) \leq \gamma_p.
$$
By construction, we have $R(p) = (1-\delta)\gamma_p < \gamma_p$ for $\gamma_p > 0$, and $R_p = 0 = \gamma_p$ when $\gamma_p = 0$; also, $R(p) = (1+\delta)\gamma_p < \gamma_p$ when $\gamma_p \in (-1/(1+\delta),0)$, and $R(p) = \gamma_p$ for $\gamma_p \in [-1,-1/(1+\delta)]$. Thus we have $R(p) < \gamma_p$ for all $p$ satisfying $\gamma_p \notin [-1,-1/(1+\delta)) \cup \{0\}$, and $R(p) \leq \gamma_p$ for all $p$.\\
We have thus shown that \eqref{eq:goal2} holds for all $p$, and furthermore it holds with a \emph{strict} inequality \emph{unless} either $\sg_p = 0$ (in which case $\gamma_p  \in \{-1,1\}$) or $\gamma_p \in [-1,-1/(1+\delta)] \cup \{0\}$. Conversely, we have seen that equality holds with $\gamma_p = 1$, and if $\gamma_p \in [-1,-1/(1+\delta)] \cup\{0\}$ then $R(p) = \gamma_p$, and thus
$$
I(p)^2 \sg_p^2 = (1-\gamma_p^2)\sg_p^2 = (1-\gamma_p^2)^2 = (1-\gamma_p)^2(1+\gamma_p)^2 = (1-\gamma_p)^2(1+R(p))^2,
$$
so equality holds in this case. This completes the proof of item (2).
\end{proof} 
\section{Key properties of $f$} \label{sec:props}
In this section we find sufficient conditions on $\delta$ and $\xi$ in order for $f = f_{\xi,\delta}$, as defined in the previous section, to satisfy the \emph{stronger} inequality \eqref{eq:LfncompIntro}. As we show below, there is a constant $\beta = \beta(\xi,\delta) > 0$ such that
\begin{equation}\label{eq:mcRDef}
\mc{R}_f(x) := \sum_{p \leq x} \frac{\text{Re}(f(p))}{p} = \sum_{p \leq x} \frac{R(p)}{p} = \beta \log\log x + O(1),
\end{equation}
so that when $g = 1\ast f$ as above, \eqref{eq:Goldforh} implies
$$
\tilde{M}_g(x) \ll \exp\left(\sum_{p \leq x} \frac{R(p)}{p}\right) = e^{\mc{R}_f(x)}.
$$
\subsection{An estimate for $\mc{R}_f(y)$}
 The purpose of this subsection is to prove the following.
 \begin{prop} \label{prop:Rx}
 Let $\delta \in (0,1/2)$, $\xi\in [-1,1]$ and let $x \geq e^{1/|\xi|}$. Set
 \begin{equation}\label{eq:phiDef}
 \phi = \phi(\delta) := -\delta\left(2-\left(1-\frac{1}{(\delta + 1)^2}\right)^{1/2}\right) \in (-2\delta,-\delta)
 \end{equation}
 Then for each $1 \leq y \leq x$ we have
 \begin{equation}\label{eq:Rsum}
 \mc{R}_f(y) = \begin{cases} (1-\delta) \log\log y + O(1) &\text{ if } y \leq e^{1/|\xi|}, \\ (1-\delta) \log(1/|\xi|) - \frac{|\phi|}{\pi} \log(|\xi|\log y) + O(1) &\text{ if } e^{1/|\xi|} < y \leq x.\end{cases}
 \end{equation}
 In particular, we have $\mc{R}_f(y) \geq \mc{R}_f(x) + O(1)$ for all $e^{1/|\xi|} < y \leq x$.
 \end{prop}
We will require the following special case of a result due to Hall and Tenenbaum (see \cite[Lem. 30.1]{HTDiv}).
\begin{lem}[Hall-Tenenbaum]\label{lem:TenEst}
Let $t \in [-1,1] \bk \{0\}$ and $z > w \geq e^{1/|t|}$. Let $\phi(u)$ be a $1$-bounded $2\pi$-periodic function of bounded variation. Then  
\begin{align} \label{eq:TenEst}
\sum_{w < p \leq z} \frac{\phi(t\log p)}{p} = \left(\frac{1}{2\pi} \int_0^{2\pi} \phi(u) du\right) \log\left(\frac{\log z}{\log w}\right) + O\left(V(\phi)\right),
\end{align}
where $V(\phi)$ denotes the total variation of $\phi$.
\end{lem}
 \begin{proof}[Proof of Proposition \ref{prop:Rx}]
Fix $1 \leq y \leq x$. 
When $y \leq e^{1/|\xi|}$ we have $\gamma_p \geq \cos(1) > 0$ for all $p \leq y$, and
$$
\gamma_p = 1+O(|\xi|^2 \log^2 p).
$$
Thus, by the prime number theorem,
%
$$
\mc{R}_f(y) = (1-\delta)\sum_{p \leq y} \frac{1}{p}+ O\left(|\xi|^2\sum_{p\leq e^{1/|\xi|}} \frac{\log^2p}{p}\right) = (1-\delta) \log \log y + O(1).
$$
Next, assume that $e^{1/|\xi|} < y \leq x$. We separate the sum over $p \leq y$ into the following subsets: 
\begin{align*}
&S_1 := \mb{P} \cap [1,e^{1/|\xi|}], \quad &S_2 := \{e^{1/|\xi|} < p \leq y : \gamma_p \geq 0\}, \\
&S_3 := \{e^{1/|\xi|} < p \leq y : 0 > \gamma_p > -\tfrac{1}{1+\delta}\}, \quad &S_4 := (\mb{P} \cap [1,y]) \bk (S_1 \cup S_2 \cup S_3).
\end{align*}
The sum over $p \in S_1$ is precisely $\mc{R}_f(e^{1/|\xi|})$, so by the previous calculation we have
$$
\sum_{p \in S_1} \frac{R(p)}{p} = (1-\delta) \log(1/|\xi|) + O(1).
$$
For the remaining ranges, we will use Lemma \ref{lem:TenEst}. Given a $2\pi$-periodic, bounded function $h$ we write
$$
I(h) := \frac{1}{2\pi} \int_{-\pi}^{\pi} h(u) du.
$$
Applying Lemma \ref{lem:TenEst} with the function 
$$
h_2(u) = (1-\delta) (\cos u)1_{\cos u \geq 0} = (1-\delta) (\cos u)1_{[-\pi/2,\pi/2]}(u),
$$
whose integral is $I(h_2) = (1-\delta)/\pi$ (and $V(h_2) = O(1)$), we may estimate the sum over $S_2$ as
$$
\sum_{p \in S_2} \frac{R(p)}{p} = \frac{1-\delta}{\pi} \log(|\xi|\log y) + O(1).
$$
For $S_3$ we use instead
$$
h_3(u) := (1+\delta) (\cos u) 1_{0 > \cos u > -\tfrac{1}{\delta + 1}}.
$$
By symmetry about $u =0$, and using the fact that $\sin(\cos^{-1}(1/(\delta+1)) = (1-1/(\delta+1)^2)^{1/2}$, we get
$$
I(h_3) = -\frac{1+\delta}{\pi} \left(1-\left(1-\frac{1}{(\delta + 1)^2}\right)^{1/2}\right)
$$
(and once again $V(h_3) = O(1)$). Therefore, by Lemma \ref{lem:TenEst},
$$
\sum_{p \in S_3} \frac{R(p)}{p} = -\frac{1+\delta}{\pi} \left(1-\left(1-\frac{1}{(\delta + 1)^2}\right)^{1/2}\right) \log(|\xi| \log y) + O(1).
$$
Finally, for $S_4$ we use 
$$
h_4(u) := (\cos u)1_{-\tfrac{1}{\delta+1} \geq \cos u \geq -1},
$$ 
for which Lemma \ref{lem:TenEst} similarly gives
$$
\sum_{p \in S_4} \frac{R(p)}{p} = -\frac{1}{\pi} \left(1-\frac{1}{(\delta+1)^2}\right)^{1/2}\log(|\xi|\log y) + O(1).
$$
Now, we observe that if $\theta := (1-1/(\delta+1)^2)^{1/2}$ then
$$
\frac{1-\delta}{\pi} - \frac{1+\delta}{\pi}(1-\theta) - \frac{\theta}{\pi} = \frac{\phi}{\pi}.
$$
Adding the contributions from the sets $S_j$ together, we obtain that for any $e^{1/|\xi|} < y \leq x$,
\begin{align} \label{eq:Rsum}
\mc{R}_f(y) &= (1-\delta) \log(1/|\xi|) + \frac{\phi}{\pi} \log(|\xi|\log y) + O(1) = (1-\delta) \log(1/|\xi|) - \frac{|\phi|}{\pi} \log(|\xi|\log y) + O(1),
\end{align}
and the first claim follows. The second claim is immediate from the first.
\end{proof}
Given $\delta \in (0,1/2)$ and $\phi = \phi(\delta)$ as in \eqref{eq:phiDef}, let us define 
\begin{equation*}
\alpha_0 = \alpha_0(\delta) := \frac{|\phi|}{\pi(1-\delta) + |\phi|}.
\end{equation*}
Note that as $0 < \delta < 1/2$ and $|\phi| < 2\delta$, we have
\begin{equation}\label{eq:alpha0Bd}
0 < \alpha_0 < \frac{4\delta}{\pi}.
\end{equation}
We obtain the following consequence from Proposition \ref{prop:Rx}. 
\begin{cor} \label{cor:largeRx}
Let $\delta \in (0,1/2)$ and $\alpha \in (\alpha_0(\delta),1)$, and set $\xi := (\log x)^{-\alpha}$. Then $\mc{R}_f(x) \gg \log\log x$.
\end{cor}
\begin{proof}
Since $x > e^{1/\xi}$, Proposition \ref{prop:Rx} yields
$$
\mc{R}_f(x) = (1-\delta) \alpha \log\log x - \frac{|\phi|}{\pi} (1-\alpha) \log\log x + O(1) = \frac{1}{\pi}\left(\alpha\left(\pi(1-\delta)+|\phi|\right)- |\phi|\right) \log\log x + O(1),
$$
with $\phi = \phi(\delta)$. As $\alpha > \alpha_0$ the bracketed expression is positive, and the claim follows.
\end{proof}
\subsection{A large peak for $L(s,f)\zeta(s)$} Our next objective is to prove the following.
\begin{prop}\label{prop:xiPeak}
There are absolute constants $c,\delta_0 > 0$ such that the following holds. Let $x$ be large and let $\sg = 1+1/\log x$. Let $\xi = (\log x)^{-\alpha}$ for some $\alpha \in (0,1/2)$ and let $\delta \in (0,\delta_0)$. Then the function $f = f_{\xi,\delta}$ constructed above satisfies
$$
|L(\sg + i\xi,f)\zeta(\sg + i\xi)| \gg (\sg-1)^{-c\delta} |L(\sg,f)|\zeta(\sg).
$$
\end{prop}
\begin{proof}
We shall consider
$$
\Delta_f(p) := \text{Re}\left((1+f(p))(p^{-i\xi}-1)\right) = I(p) \sg_p - (1-\gamma_p)(1+R(p)), \quad p \leq x.
$$
By Mertens' theorem, we have (as in the derivation of \eqref{eq:realpart})
\begin{align}
\left|\frac{L(\sg+i\xi,f)\zeta(\sg+i\xi)}{L(\sg,f)\zeta(\sg)}\right| \asymp \exp\left(\sum_{p \leq x} \frac{\text{Re}(f(p)p^{-i\xi} + p^{-i\xi} - f(p)-1)}{p}\right) &= \exp\left(-\sum_{p \leq x} \frac{\text{Re}((1+f(p))(1-p^{-i\xi}))}{p}\right) \nonumber \\
&= \exp\left(\sum_{p \leq x} \frac{\Delta_f(p)}{p}\right). \label{eq:toDelta}
\end{align}
By Proposition \ref{prop:consf}(2), we have $\Delta_f(p) \geq 0$ for all $p$. 
Thus, in the notation of the proof of Proposition \ref{prop:Rx} (taking $y = x$), it suffices to consider only $p \in S_2$. \\
Fix $p \in S_2$ for the moment. Using $\Delta_f(p) \geq 0$, we see that
$$
I(p)\sg_p + (1+R(p))(1-\gamma_p) = 2I(p)\sg_p - \Delta_f(p) \leq 2I(p)\sg_p \leq 2.
$$
It follows that
\begin{align} \label{eq:toSq}
2\Delta_f(p) \geq I(p)^2 \sg_p^2 - (1+R(p))^2(1-\gamma_p)^2.
\end{align}
As $p \in S_2$ we have $R(p) = (1-\delta)\gamma_p$, thus
$$
I(p)^2\sg_p^2 = \sg_p^2(1-R(p)^2) = \sg_p^2(\sg_p^2 + 2\delta \gamma_p^2) + O(\delta^2).
$$
Moreover, we have
\begin{align*}
(1+R(p))^2(1-\gamma_p)^2 &= (1-\delta \gamma_p -(1-\delta)\gamma_p^2)^2 = 1 + (1-\delta)^2\gamma_p^4 + 2(\delta \gamma_p^3 - \delta \gamma_p - (1-\delta)\gamma_p^2) + O(\delta^2) \\
&= (1-2\gamma_p^2 + \gamma_p^4) + 2\delta\left(\gamma_p^2(1-\gamma_p^2) - \gamma_p(1-\gamma_p^2)\right)  + O(\delta^2) \\
&= \sg_p^4 + 2\delta \gamma_p^2 \sigma_p^2 - 2\delta \gamma_p \sigma_p^2 + O(\delta^2).
\end{align*}
Combining these two estimates with \eqref{eq:toSq}, we deduce that for each $p \in S_2$,
$$
\Delta_f(p) \geq \max\{0,\delta \gamma_p \sigma_p^2 + O(\delta^2)\}.
$$
We may restrict further to the set
$$
S' := \{e^{1/|\xi|} < p \leq x : \ \gamma_p \in [\tfrac{1}{10}, \tfrac{\sqrt{99}}{10}]\} \subseteq S_2,
$$
say, on which we have $\min\{\gamma_p,|\sg_p|\} \geq \tfrac{1}{10}$. Since the set 
$$
\{u \in [-\pi,\pi] : \ \cos u \in [\tfrac{1}{10}, \tfrac{\sqrt{99}}{100}]\}
$$ 
is an interval whose measure $c_0$ is a positive absolute constant (and whose indicator function is thus of variation bounded by an absolute constant), it follows from Lemma \ref{lem:TenEst} that
$$
\sum_{p \in S'} \frac{\Delta_f(p)}{p} \geq \left(\frac{\delta}{1000} + O(\delta^2)\right) \frac{c_0}{2\pi} \log(|\xi|\log x)  + O(1) \geq \frac{c_0(1-\alpha)\delta}{7000}\left(1+O(\delta)\right)  \log\log x + O(1).
$$
Now, there must exist some $\delta_0$ such that if $0 <\delta < \delta_0$ then the bracketed term $1+O(\delta) \geq 1/2$. Since $\alpha < 1/2$ by assumption, choosing $c := c_0/28000$ then yields 
$$
\sum_{p \leq x} \frac{\Delta_f(p)}{p} \geq \sum_{p \in S'} \frac{\Delta_f(p)}{p} \geq c\delta \log\log x + O(1) = c\delta \log\left(\frac{1}{\sg-1}\right) + O(1).
$$
Inserting this into the RHS of \eqref{eq:toDelta} completes the proof.
\end{proof}
%
\section{Proof of Theorem \ref{thm:counter}} \label{sec:proof}
In this section we shall prove Theorem \ref{thm:counter}. We need a few more lemmas. The first is a special case of a well-known bound of Halberstam and Richert \cite{HalRic}.
\begin{lem} \label{lem:HR}
Let $f \in \mc{F}$. Then for any $X \geq 2$,
$$
|\tilde{M}_{1\ast f}(X)| \ll \log X.
$$
\end{lem}
\begin{proof}
Since $|1\ast f(n)| \leq d(n)$ for all $n$, the main result of \cite{HalRic} and the triangle inequality yield
$$
|\tilde{M}_{1\ast f}(X)| \leq \tilde{M}_{|1\ast f|}(X) \ll \exp\left(\sum_{p \leq X} \frac{|1+f(p)| - 1}{p}\right) \ll \log X,
$$
as claimed.
\end{proof}
The following result is standard.
\begin{lem} \label{lem:Planch}
Let $x \geq 3$, let $\sg = 1+1/\log x$. Then for any $f \in \mc{F}$,
\begin{equation*}
\int_{-\infty}^{\infty} \frac{|L(\sg + it,1\ast f)|^2}{\sg^2 + t^2} dt \asymp \int_0^{\infty} |\tilde{M}_{1\ast f}(e^v)|^2 e^{-2(\sg-1)v} dv.
\end{equation*}
\end{lem}
\begin{proof}
For $v \in \mb{R}$ we define the function
$$
G(v) := \tilde{M}_{1\ast f}(e^v) e^{-(\sg-1) v} = M_{1\ast f}(e^v) e^{-\sg v}.
$$ 
Clearly, $G(v) = 0$ for $v < 0$, and by Lemma \ref{lem:HR},
$$
|G(v)| \ll v e^{-(\sg-1)v},
$$ 
so that $G \in L^1(\mb{R}) \cap L^2(\mb{R})$. It is then readily verified that $G$ has Fourier transform
$$
\hat{G}(t) = \int_{-\infty}^{\infty} G(v) e^{-itv} dv = \frac{L(\sg + it,1\ast f)}{\sg + it}, \quad t\in\mb{R}.
$$
The claim now follows by Plancherel's theorem.
\end{proof}
\begin{proof}[Proof of Theorem \ref{thm:counter}]
Assume for the sake of contradiction that there is a constant $C > 0$ such that for every $f \in \mc{F}$ and every $y \geq 1$,
$$
|L_f(y)| \leq C\left(1+(\log y)e^{-\mb{D}(f,1;y)^2}\right).
$$
Then it follows as in \eqref{eq:Goldforh} that there is a further constant $C_1$ (depending on $C$) such that for any $f \in \mc{F}$ and $y \geq 1$,
\begin{equation}\label{eq:M1fy}
|\tilde{M}_{1\ast f}(y)| \leq C_1\left(1+\exp\left(\sum_{p \leq y} \frac{\text{Re}(f(p))}{p}\right)\right).
\end{equation}
We now select a scale $x$ to be chosen sufficiently large in an absolute sense, and set $\sg = 1 + 1/\log x$. We also let 
$$
\delta \in (0, \min\{\delta_0,\pi/12\}), \quad \alpha \in (\alpha_0(\delta),1/3),
$$
where $\delta_0,\alpha_0(\delta)$ are as defined in Section \ref{sec:props}, and 
$$
\xi := (\log x)^{-\alpha} \in [-1/3,1/3].
$$ 
We note that the range $\alpha \in (\alpha_0,1/3)$ is non-empty in light of \eqref{eq:alpha0Bd} and the assumption $\delta < \pi/12$.\\ Applying Lemma \ref{lem:Planch}, restricting to the interval $t \in [-1/2,1/2]$ on the Fourier side and making the change of variables $y = (\sg-1)v$ on the phase side, we obtain
\begin{align} \label{eq:Bessel}
\int_{-1/2}^{1/2} |L(\sg + it,f)|^2|\zeta(\sg + it)|^2 dt &\asymp \int_{-1/2}^{1/2} \frac{|L(\sg + it,1\ast f)|^2}{\sg^2 + t^2} dt \ll \int_0^{\infty} |\tilde{M}_{1\ast f}(e^v)|^2 e^{-2(\sg-1) v}dv \\
&= \frac{1}{\sg-1} \int_0^{\infty} |\tilde{M}_{1\ast f}(x^y)|^2 e^{-2y} dy. \nonumber
\end{align}
By Mertens' theorems, for any $1$-bounded sequence $(a_p)_{p}$ it holds that whenever $|t-\xi| \leq 1/\log x$,
$$
\sum_p \frac{a_p}{p^{\sg+it}} = \sum_{p \leq x} \frac{a_p}{p^{1+it}} + O(1) = \sum_{p \leq x} \frac{a_p}{p^{1+i\xi}} + O(1) = \sum_p \frac{a_p}{p^{\sg + i\xi}} + O(1).
$$
Applying this with $a_p = f(p)$ and $a_p = 1$, and restricting the $t$-integral further to the interval $[\xi-1/\log x, \xi + 1/\log x]$, the LHS in \eqref{eq:Bessel} is
$$
\geq \int_{\xi-\tfrac{1}{\log x}}^{\xi + \tfrac{1}{\log x}} |L(\sg + it,f)|^2 |\zeta(\sg + it)|^2 dt \asymp (\sg-1) |L(\sg + i\xi,f)|^2 |\zeta(\sg + i\xi)|^2.
$$
Since $\delta \in (0,\delta_0)$, Proposition \ref{prop:xiPeak} implies that there is an absolute constant $c > 0$ such that this is
$$
\gg (\sg-1)^{1-2c\delta} |L(\sg,f)|^2 \zeta(\sg)^2 \asymp (\sg-1)^{-1-2c\delta} \exp\left(2\sum_{p \leq x} \frac{\text{Re}(f(p))}{p}\right) = (\sg-1)^{-1-2c\delta} e^{2\mc{R}_f(x)},
$$
with $\mc{R}_f$ defined as in \eqref{eq:mcRDef}. On combining this with \eqref{eq:Bessel} and rearranging, we get that
\begin{equation}\label{eq:BessRef}
(\sg-1)^{-2c\delta} e^{2\mc{R}_f(x)} \ll \int_0^{\infty} |\tilde{M}_{1\ast f}(x^y)|^2e^{-2y} dy.
\end{equation}
Next, we truncate the integral on the RHS of \eqref{eq:BessRef} from above. By Lemma \ref{lem:HR}, we obtain for each $y \geq 1$ that
$$
|\tilde{M}_{1\ast f}(x^y)| 
\ll y(\log x).
$$ 
Squaring this and inserting it into the integral for $y \geq 4\log\log x$, we see (crudely) that
\begin{equation}\label{eq:Tail}
\int_{4\log\log x}^{\infty} |\tilde{M}_{1\ast f}(x^y)|^2 e^{-2y} dy \ll (\log x)^2 \int_{10\log\log x}^{\infty} y^2 e^{-2y} dy \ll (\sg-1)^2 \ll e^{2\mc{R}_f(x)}.
\end{equation}
Next, on the range $1 \leq y \leq 4\log\log x$ we observe that 
$$
|\mc{R}_f(x^y) - \mc{R}_f(x)| \leq \sum_{x < p \leq x^y} \frac{1}{p} = \log y + O\left(\frac{1}{\log x}\right) = \log\log\log x + O(1).
$$
Since $\alpha > \alpha_0(\delta)$, by Corollary \ref{cor:largeRx} we have that $\mc{R}_f(x) \gg \log\log x$, and so if $x$ is sufficiently large then, also, 
$$
\min_{1 \leq y \leq 4\log\log x} \mc{R}_f(x^y) \gg \log\log x.
$$ 
We therefore deduce using \eqref{eq:M1fy} that
$$
\int_1^{4\log\log x} |\tilde{M}_{1\ast f}(x^y)|^2 e^{-2y} dy \ll \max_{1 \leq y \leq 4\log\log x} |\tilde{M}_{1\ast f}(x^y)|^2 \ll e^{2\mc{R}_f(x)} \log^2\left(\frac{1}{\sg-1}\right),
$$
where the implicit constant now depends on $C$. Combining this with \eqref{eq:Tail} and inserting it into \eqref{eq:BessRef}, when $x$ is sufficiently large we obtain
\begin{equation}\label{eq:BesselFiner}
(\sg-1)^{-c\delta} e^{2\mc{R}_f(x)} \ll \int_0^1 |\tilde{M}_{1\ast f}(x^y)|^2 dy.
\end{equation}
Finally, we treat the integral over $0 \leq y \leq 1$ on the RHS by splitting it into length $\tfrac{1}{\log x}$ segments. Observe first that in view of \eqref{eq:triv}, for any $Z > e$ we have
\begin{equation}\label{eq:Lipfor1f}
\max_{Z/e < z \leq Z} |\tilde{M}_{1\ast f}(z) - \tilde{M}_{1\ast f}(Z)|  = O(1) + \max_{Z/e < z \leq Z} |L_f(z)-L_f(Z)| \ll 1. 
\end{equation}
It follows that 
$$
\max_{j/\log x \leq y <(j+1)/\log x} |\tilde{M}_{1\ast f}(x^y)|^2 \ll |\tilde{M}_{1\ast f}(e^j)|^2 + 1, \quad 0 \leq j \leq \log x.
$$
Thus, the RHS in \eqref{eq:BesselFiner} is
\begin{align*}
\ll \frac{1}{\log x} \sum_{0 \leq j \leq \log x} |\tilde{M}_{1\ast f}(e^j)|^2 + 1 &= \frac{1}{\log x} \sum_{0 \leq j \leq 1/|\xi|} |\tilde{M}_{1\ast f}(e^j)|^2 + \frac{1}{\log x} \sum_{1/|\xi| < j \leq \log x} |\tilde{M}_{1\ast f}(e^j)|^2 + 1 \\
&=: T_1 + T_2 + 1.
\end{align*}
Because it will be useful in the next section, we isolate a further range from $0 \leq j \leq 1/|\xi|$. Set 
\begin{equation}\label{eq:betaDef}
\beta := \alpha\left(1+\frac{|\phi|}{\pi} - \delta\right) - \frac{|\phi|}{\pi} \in (0,\alpha),
\end{equation}
and let $J_0 := \min\{1/|\xi|, (\log x)^{(1+2\beta)/3}\}$. We now divide $T_1 = T_{1,s} + T_{1,\ell}$, where
$$
T_{1,s} := \frac{1}{\log x} \sum_{0 \leq j \leq J_0} |\tilde{M}_{1\ast f}(e^j)|^2, \quad T_{1,\ell} := \frac{1}{\log x} \sum_{J_0 < j \leq 1/|\xi|} |\tilde{M}_{1\ast f}(e^j)|^2.
$$
Applying Lemma \ref{lem:HR}, noting that from the proof of Corollary \ref{cor:largeRx}, 
$$
\mc{R}_f(x) = \beta \log\log x + O(1),
$$ 
we find that
$$
T_{1,s} \ll \frac{1}{\log x} \sum_{0 \leq j \leq J_0} j^2 \ll \frac{J_0^3}{\log x} \ll (\log x)^{2\beta} \asymp e^{2\mc{R}_f(x)}.
$$
We next treat $T_{1,\ell}$ and $T_2$. 
In both cases we use \eqref{eq:M1fy} 
together with Proposition \ref{prop:Rx}
to get (with a suitably large implicit constant that depends on $C$)
\begin{equation}\label{eq:M1fbyj}
|\tilde{M}_{1\ast f}(e^j)| \ll e^{\mc{R}_f(e^j)} \ll \begin{cases} j^{1-\delta} &\text{ if } J_0 \leq j \leq |\xi|^{-1}, \\ 
								    |\xi|^{-(1-\delta +\tfrac{|\phi|}{\pi})} j^{-\tfrac{|\phi|}{\pi}} &\text{ if } |\xi|^{-1} < j \leq \log x. \end{cases}
\end{equation}
With this estimate, we thus see that
$$
T_{1,\ell} \ll (\sg-1) \sum_{J_0 < j \leq |\xi|^{-1}} j^{2(1-\delta)} \ll (\sg-1) |\xi|^{-(3-2\delta)} = (\sg-1)^{1-(3-2\delta)\alpha},
$$
and also that
$$
T_2  \ll (\sg-1) |\xi|^{-2(1-\delta+\tfrac{|\phi|}{\pi})} \sum_{|\xi|^{-1} < j \leq \tfrac{1}{\sg-1}} j^{-\tfrac{2|\phi|}{\pi}}.
$$
Recalling that $|\phi| \in [\delta,2\delta]$ and assuming that $\delta < 1/2$, say, we get that (on invoking Proposition \ref{prop:Rx} again)
$$
T_2 \ll (\sg-1)|\xi|^{-2(1-\delta+\tfrac{|\phi|}{\pi})} \cdot (\sg-1)^{-1+\tfrac{2|\phi|}{\pi}} \ll \exp\left(2(1-\delta) \log(1/|\xi|) - \frac{2|\phi|}{\pi} \log(|\xi|\log x)\right) \ll e^{2\mc{R}_f(x)}.
$$
Since $\alpha < 1/3$, it follows that
$$
(\sg-1)^{-2c\delta} e^{2\mc{R}_f(x)} \ll T_{1,s} + T_{1,\ell} + T_2 +1 \ll (\sg-1)^{1-(3-2\delta)\alpha} + e^{2\mc{R}_f(x)} +1 \ll e^{2\mc{R}_f(x)},
$$
where the implicit constant depends on $C$. Since $c\delta > 0$, this gives a contradiction as soon as $x$ is sufficiently large as a function of $C$. This proof is now complete.
\end{proof}
\begin{rem} \label{rem:counter}
Given the scale $x$, the proof above only used the assumed bound \eqref{eq:M1fy} to treat $\tilde{M}_{1\ast f}(X)$ when 
$$
\exp\left((\log x)^{(1+2\beta)/3}\right) < X \leq x^{4\log\log x}.
$$
Thus, whenever $x \geq x_0(C)$ there must exist some scale $X = x^\ast$ in this range for which \eqref{eq:M1fy} fails to hold. We will use this crucially in the next section.
\end{rem}
\section{Proof of Theorem \ref{thm:infScales}} \label{sec:infScales}
In this section we will show how the proof of Theorem \ref{thm:counter} can be used to construct a single function $f \in \mc{F}$ such that $L_f(x)$ fails the bound of Conjecture \ref{conj:Gold} infinitely often, thereby proving Theorem \ref{thm:infScales}. \\
\begin{proof}[Proof of Theorem \ref{thm:infScales}]
Throughout the proof we fix $\delta \in (0, \min\{\delta_0,\pi/12\})$ and $\alpha \in (\alpha_0(\delta),1/3)$, where $\delta_0$ and $\alpha_0$ are as in Section \ref{sec:props}. Let $(a_k)_k$ be any increasing, unbounded sequence. We will construct a function $f \in \mc{F}$ inductively so that there exists a sequence of scales $(x_k^\ast)_k$ on which 
\begin{align} \label{eq:Lfk}
|L_f(x_k^\ast)| \geq a_k \left(1 + \exp\left(\sum_{p \leq x_k^\ast} \frac{\text{Re}(f(p))}{p}\right)\right) \text{ for all } k \geq 1.
\end{align}
For the base case $k = 1$, we select $x_1 \geq x_0(a_1)$ and define $\xi_1 := (\log x_1)^{-\alpha}$. The proof of Theorem \ref{thm:counter} and Remark \ref{rem:counter} imply that there exists a scale $x_1^\ast \in [\exp((\log x)^{(1+2\beta)/3}), x_1^{4\log\log x_1}]$ on which the function $f_1 = f_{\xi_1,\delta}$ must satisfy
$$
|L_f(x_1^{\ast})| 
\geq a_1\left(1+\exp\left(\sum_{p \leq x_1^\ast} \frac{\text{Re}(f(p))}{p}\right)\right).
$$
Now, assume that there is $K \geq 1$, functions $f_1,\ldots,f_K$ and scales $x_1^{\ast},\ldots, x_K^\ast$ such that \eqref{eq:Lfk} holds with $f = f_k$, for each $1 \leq k \leq K$. We select $x_{K+1} > e^{x_K^\ast}$, eventually chosen sufficiently large in terms of $a_{K+1}$, and $\xi_{K+1} := (\log x_{K+1})^{-\alpha}$, and define $f_{K+1} \in \mc{F}$ on primes via
$$
f_{K+1}(p) := \begin{cases} f_K(p) &\text{ if } p \leq x_K^\ast \\ f_{\xi_{K+1},\delta}(p) &\text{ if } p > x_K^\ast. \end{cases}
$$
We will show that $f_{K+1}$ will also fail to satisfy Conjecture \ref{conj:Gold} with $C = a_{K+1}$, when $x_{K+1}$ is sufficiently large. As in the previous section we will assume for contradiction that this is not so, and will follow the proof of Theorem \ref{thm:counter} above (and hereafter the implicit constants may depend on $a_{K+1}$). \\
Let $\sg_{K+1} := 1 + 1/\log x_{K+1}$. In the notation of Proposition \ref{prop:xiPeak}, note that
$$
\left|\sum_{p \leq x_{K+1}} \frac{\Delta_{f_{K+1}}(p)}{p} -\sum_{p \leq x_{K+1}} \frac{\Delta_{f_{\xi_{K+1},\delta}}(p)}{p}\right| \leq 4\sum_{p \leq x_K^\ast} \frac{1}{p} < 4\log\log \log x_{K+1}.
$$
It follows that
$$
\left|\frac{L(\sg_{K+1} +i \xi_{K+1},1\ast f_{K+1})}{L(\sg_{K+1}, 1\ast f_{K+1})}\right| \gg \frac{1}{\log^4(\tfrac{1}{\sg_{K+1}-1})} \left|\frac{L(\sg_{K+1} +i \xi_{K+1}, 1\ast f_{\xi_{K+1},\delta})}{L(\sg_{K+1},1\ast f_{\xi_{K+1},\delta})}\right|
$$
Proposition \ref{prop:xiPeak} therefore implies that
\begin{align} \label{eq:xiPeakK}
|L(\sg_{K+1} + i\xi_{K+1},f_{K+1}) \zeta(\sg_{K+1}+i\xi_{K+1})| &\gg \frac{(\sg_{K+1}-1)^{-c\delta}}{\log^4(\tfrac{1}{\sg_{K+1}-1})} |L(\sg_{K+1}, f_{K+1})|\zeta(\sg_{K+1}) \nonumber\\
&\geq (\sg_{K+1}-1)^{-c\delta/2} |L(\sg_{K+1},f_{K+1})|\zeta(\sg_{K+1})
\end{align}
on taking $x_{K+1}$ slightly larger if needed.  \\
Similarly, we have that, uniformly over $y \leq x_{K+1}$,
\begin{equation}\label{eq:fKtofxiK}
\left|\mc{R}_{f_{K+1}}(y) - \mc{R}_{f_{\xi_{K+1},\delta}}(y)\right| \leq \sum_{p \leq x_K^\ast} \frac{1}{p} < \log\log\log x_{K+1} + O(1) = \log\log\left(\frac{1}{\sg_{K+1}-1}\right) + O(1).
\end{equation}
Now, following the proof of Theorem \ref{thm:counter} up to \eqref{eq:BessRef}, but using \eqref{eq:xiPeakK} in place of Proposition \ref{prop:xiPeak}, we see that
$$
(\sg_{K+1}-1)^{-c\delta} e^{2\mc{R}_{f_{K+1}}(x_{K+1})} \ll \int_0^{\infty} |\tilde{M}_{1\ast f_{K+1}}(x_{K+1}^y)|^2 e^{-2y} dy.
$$
We handle the ranges $y \geq 4 \log\log x_{K+1}$ and $1 \leq y \leq 4\log\log x_{K+1}$ as above, getting
$$
\int_1^{4 \log\log x_{K+1}} |\tilde{M}_{1\ast f_{K+1}}(x_{K+1}^y)|^2 e^{-2y} dy \ll e^{2\mc{R}_{f_{K+1}}(x)} \log^2\left(\frac{1}{\sg_{K+1}-1}\right) + (\sg_{K+1}-1)^2.
$$
We thus deduce that
$$
(\sg_{K+1}-1)^{-c\delta} e^{2\mc{R}_{f_{K+1}}(x_{K+1})} \ll \int_0^1 |\tilde{M}_{1\ast f_{K+1}}(x_{K+1}^y)|^2 dy.
$$
We may repeat the analysis of the range $[0,1]$ as well, but combine \eqref{eq:fKtofxiK} with \eqref{eq:M1fbyj} (which multiplies all of the bounds by $\log^2(\tfrac{1}{\sg_{K+1}-1})$) to obtain instead that
\begin{align*}
\int_0^1 |\tilde{M}_{f_{K+1}}(x_{K+1}^y)|^2 dy &\ll \log^2\left(\frac{1}{\sg_{K+1}-1}\right)\left(T_1 + T_2 + 1\right) \ll \log^2\left(\frac{1}{\sg_{K+1}-1}\right) e^{2\mc{R}_{f_{\xi_{K+1},\delta}}(x_{K+1})} \\
&\ll \log^{4}\left(\frac{1}{\sg_{K+1}-1}\right) e^{2\mc{R}_{f_{K+1}(x_{K+1})}}.
\end{align*}
If $x_{K+1}$ is sufficiently large then again this yields a contradiction. Now, by Remark \ref{rem:counter}, there must be a point\footnote{The $o(1)$ term here arises from the use of \eqref{eq:fKtofxiK}, which again changes $\beta \log\log x_{K+1}$ by $O(\log\log\log x_{K+1})$ at most.} 
$$
x_{K+1}^{\ast} \in \left[\exp((\log x_{K+1})^{(1+2\beta + o(1))/3}), x_{K+1}^{4\log\log x_{K+1}}\right]
$$ 
at which
$$
|L_{f_{K+1}}(x_{K+1}^\ast)| \geq a_{K+1} \left(1+\exp\left(\sum_{p\leq x_{K+1}^{\ast}} \frac{\text{Re}(f_{K+1}(p))}{p}\right)\right).
$$
Now since $x_{K+1} > e^{x_K^\ast}$ and $\beta > 0$, we have that $x_{K+1}^{\ast} \geq e^{(x_K^\ast)^{1/3-o(1)}} > x_K^\ast$. \\
This completes the inductive step. Thus, by induction we have produced an infinite sequence of functions $(f_k)_k \subset \mc{F}$ and an infinite increasing sequence $(x_k^\ast)_k$ such that
\begin{enumerate}[(i)]
\item for every $k \geq 1$, $f_{k+1}|_{[1,x_k^{\ast}]} = f_k$, 
\item for every $k \geq 1$, 
$$
|L_{f_k}(x_k^\ast)| \geq a_k \left(1 + \exp\left(\sum_{p \leq x_k^\ast} \frac{\text{Re}(f_k(p))}{p}\right)\right).
$$
\end{enumerate}
We may therefore define $f \in \mc{F}$ compatibly on primes via $f(p) = f_K(p)$ whenever $p \leq x_K^\ast$, from which \eqref{eq:Lfk} holds.
 Since $a_k \ra \infty$ as $k \ra \infty$. the claim follows.
\end{proof}

\bibliographystyle{plain}
\bibliography{GoldConj.bib}

\end{document}